\documentclass[11pt]{amsart}
\usepackage{cancel}
\usepackage{amsmath}
\usepackage{amsthm}
\usepackage{amssymb}
\usepackage{amsfonts,mathrsfs}
\usepackage{stmaryrd}
\usepackage{amsxtra}  
\usepackage{epsfig}
\usepackage{verbatim}
\usepackage[all]{xy}
\usepackage{graphicx}

\usepackage{tikz}

\theoremstyle{plain}
\newtheorem{theorem}[equation]{Theorem}
\newtheorem{proposition}[equation]{Proposition}
\newtheorem{lemma}[equation]{Lemma}

\theoremstyle{remark}
\newtheorem{question}[equation]{Question}
\newtheorem{remark}[equation]{Remark}
\numberwithin{equation}{section}
\newtheorem{example}[equation]{Example}

\newcommand{\dbar}{\bar \partial}

\newcommand{\re}{\text{Re}}
\newcommand{\im}{\text{Im}}

\newcommand\littleO{
  \mathchoice
    {{\scriptstyle\mathcal{O}}}
    {{\scriptstyle\mathcal{O}}}
    {{\scriptscriptstyle\mathcal{O}}}
    {\scalebox{0.6}{$\scriptscriptstyle\mathcal{O}$}}
  }

\newcommand{\sh}{{\mathscr H}}

\newcommand{\sL}{{\mathscr L}}

\newcommand{\C}{{\mathbb C}}

\newcommand{\R}{{\mathbb R}}

\begin{document}

\title[]{Local Plurisubharmonic Defining Functions on the Boundary}
\author{Luka Mernik}
\address{Department of Mathematics, \newline The Ohio State University, Columbus, Ohio, USA}
\email{mernik.1@osu.edu}
 
 
\begin{abstract}
Necessary conditions for a domain $\Omega\subset \C^n$ admitting a local plurisubharmonic defining function on the boundary are given. In tandem, we give an algorithm to construct a local plurisubharmonic defining function on the boundary when one exists. In some cases we show that the necessary conditions are also sufficient.
\end{abstract}

\maketitle 

\section{Introduction}

A strongly pseudoconvex domain in $\C^n$ always admits a (strongly) plurisubharmonic defining function \cite{morrowrossi1975}. However weakly pseudoconvex domains do not generally satisfy an analogous property. Diederich and Fornaess \cite{DieFor77-1} and Fornaess \cite{Fornaess79} found examples of weakly pseudoconvex domains in $\C^2$ which do not admit plurisubharmonic defining functions.

A weaker notion than having a plurisubharmonic defining function was also introduced by Diederich and Fornaess \cite{DieFor77-2}. They showed that for any bounded pseudoconvex domain $\Omega$ in $\C^n$ with $C^2$ boundary, there exists a positive constant $\eta$ and a defining function $\rho$ such that $-(-\rho)^\eta$ is plurisubharmonic on $\Omega$. Such $\eta$ is called a Diederich-Fornaess exponent and the supremum of all Diederich-Fornaess exponents is the Diederich-Fornaess index. The existence of bounded plurisubharmonic functions have later been generalized to domains with $C^1$ boundary \cite{KerzmanRosay1981} and Lipschitz boundary \cite{Demailly1987}, \cite{Harrington07}.

If a domain admits a plurisubharmonic defining function then the Diederich-Fornaess index is $1$. 
Furthermore, Fornaess and Herbig \cite{herbigfornaess07}, \cite{herbigfornaess08} showed that if a domain has a defining function plurisubharmonic on the boundary then the Diederich-Fornaess exponent is $1$. However the converse is not true: a domain with the Diederich-Fornaess index $1$ need not admit a plurisubharmonic defining function as shown by Behrens' example \cite{Behrens85}.
Examples with Diederich-Fornaess exponent strictly less than $1$ include worm domains; for explicit computations see \cite{Liu19-1}.

The focus of this paper is to identify obstructions to having a local plurisubharmonic defining function on the boundary. 
In addition to Diederich and Fornaess examples, Behrens \cite{Behrens85}  constructed a pseudoconvex domain of D'Angelo type $6$ in $\C^2$ which does not admit a local plurisubharmonic defining function, not even on the boundary. The obstruction in their examples comes from the inability to construct a real-valued multiplier function $h$ that satisfies a specific partial differential equation on certain curves lying in the boundary. 
We generalize this statement and make it precise by stating the explicit differential equation that needs to be solved. Namely:
\begin{theorem}
Suppose that $\rho=r\cdot h$ is a local plurisubharmonic defining function. Then 
\[\frac{\partial}{\partial z} \log h = -\frac{\partial}{\partial z} \log r_{\bar w} +E , \] where $|E|^2\leq C(\sL_r+|r_z|^2)$.
\end{theorem}

The second obstruction involves the determinant of the complex Hessian $\sh$. We show that the Levi form generates ``free'' positivity for the determinant of the complex Hessian. This gives rise to a necessary and sufficient condition for having a plurisubharmonic defining function. 
\begin{theorem}
A domain $\Omega=\{r<0\}\subset \C^2$ has a local plurisubharmonic defining function in a neighborhood of $p\in b\Omega$ if and only if there exists a function $T:\C^2\rightarrow \R$ and a positive constant $C>0$ such that $-\sh_{(1+T)r}\leq C \sL_r$ in a neighborhood of $p$.
\end{theorem}
Thus the Levi form is a threshold that a determinant of the complex Hessian of $h$ needs to achieve in order for $\rho=r\cdot h$ to be a local plurisubharmonic defining function on the boundary.
``Spreading'' of positivity of the Hessian and other intermediate positivity conditions have been studied in \cite{HerMcN09} and  \cite{HerMcN2012}.

The last possible obstruction is the $|r_z|^2$ term. It is unknown whether this is actually an obstacle for producing a local plurisubharmonic defining function on the boundary and if so, what a geometric interpretation should be.  The Example \ref{Ex:type4} shows that this term needs to be considered, but does not necessarily prevent the domain from having a local plurisubharmonic defining function on the boundary. 

For domains in $\C^n$ the above conditions are still necessary as each $2\times2$ principal minor of the complex Hessian needs to be positive semidefinite on its own. However a few more obstructions arise. The same multiplier function must satisfy the corresponding differential equations for $z_1,...,z_{n-1}$ simultaneously. The necessary conditions imposed by positive semidefiniteness of higher order minors of the complex Hessian have not been studied.

Kohn \cite{Kohn99} showed that if the Diederich-Fornaess index is $1$ then the $\dbar$-Neumann operator is globally regular. For more general necessary condition needed for domains to have index $1$ see \cite{Liu19-2}. In general, plurisubharmonic defining functions are of importance in the study of the $\dbar$-Neumann operator as it is a sufficient condition for global regularity of $\dbar$-Neumann operator and the Bergman projection \cite{BoasStraube90}, \cite{BoaStr91}, \cite{BoasStraube1993}, \cite{BoasStraube99}. That is, the domains with plurisubharmonic defining function satisfy condition $R$.

This paper is a part of the authors PhD thesis. The author would like to thank his adviser Jeffery D. McNeal for bringing this problem to his attention and for many helpful discussions.

\section{Definitions and Notation}

We denote the partial derivatives with a subscript, e.g. $r_{z_j} =\frac{\partial r}{\partial z_j}$.

We say that $r$ is a defining function for a domain $\Omega\subset \C^2$ provided that  $\Omega=\{(z,w)\in \C^2 : r(z,w)<0\}\subset \C^2$ and $\nabla r\neq 0$ on the boundary.

Let $p\in b\Omega$.
By rotation and translation we may assume that $p=(0,0)$ is the origin. By the Implicit function theorem we may assume that  $r$ is of the form $$r(z,w)=\im w+ F(z,\bar z, \re w,\im w) \ ,$$ where the terms in $F$ are of degree at least $2$. The crucial property is that $r_w\neq0$ is non-vanishing in some neighborhood of the origin.

The complex tangent space $T^{1,0}_p b\Omega$  to $b\Omega$ at $p$ is defined to be
all $\xi=(\xi_1,\xi_2)\in\C^2$ satisfying $$\sum_{j=1}^2\frac{\partial r}{\partial z_j}(p)\xi_j=0.$$
$\Omega$ is pseudoconvex if

$$\sL_r(\xi)=\sum_{j,k=1}^2\frac{\partial^2 r}{\partial z_j\partial\bar z_k}(p)\xi_j\bar\xi_k\geq 0\qquad\text{ for all } \xi\in T^{1,0}_pb\Omega,$$
for all $p\in b\Omega$. 
In $\C^2$ the tangent space is $1$-dimension and $\langle r_w , -r_z\rangle\in T^{1,0}_pb\Omega$, so define 
$$\sL_r= \sL_r(\langle r_w , -r_z\rangle)=r_{z\bar z}|r_w|^2 +r_{w\bar w}|r_z|^2 -2\re[r_{z\bar w}r_w r_{\bar z} ] .$$

A function $r$ is plurisubharmonic if 
$$\sum_{j,k=1}^2\frac{\partial^2 r}{\partial z_j\partial\bar z_k}(p)\xi_j\bar\xi_k\geq 0\qquad\text{ for all } \xi\in \C^2 . 
$$
Alternative description of plurisubharmonic function $r$ is that a complex Hessian matrix $H_r$ is positive semidefinite. This is equivalent to the diagonal entries $r_{z\bar z},r_{w\bar w}\geq0$ and the determinant $\sh_r=det(H_r)=r_{z\bar z}r_{w\bar w}-|r_{z\bar w}|^2\geq0$.

We will consider the local versions of the above definitions, i.e. the above definitions only need to hold on a neighborhood of the origin inside the boundary.

Throughout, $r$ denotes the defining function of $\Omega$, and $\rho$ denotes a local plurisubharmonic defining function, provided one exists.
In that case we can write $\rho=r\cdot h$, for some real-valued, multiplier function $h$, which does not vanish in a neighborhood of the origin.

Finally, we use $\mathcal O$ and $\littleO$ as big O and little O notation respectively.

\section{Necessary Conditions}

Let $\Omega\subset \C^2$ be a domain with a defining function $r(z,w)= \im w +F(z,\bar z, \re w,\im w)$ where the terms in $F$ are of order at least $2$. Suppose that $\Omega$ admits a local plurisubharmonic defining function on the boundary $\rho=r\cdot h$. Write $h=1+Kr+T$ where $K\in \R$ and $T$ is a real-valued function.

We wish to analyze the properties of $\rho$, derivatives of $\rho$, and the complex Hessian of $\rho$ in terms of $r$ and the complex Hessian of $r$.

The first key proposition is 
\begin{proposition}\label{Prop:Hessian}
\[\sh_\rho=\sh_{hr}=\sh_{(1+Kr+T)r}=2Kh\sL_r+\sh_{(1+T)r}\]
\end{proposition}

\begin{proof}
Let $h=1+Kr+T$ and $p=1+T$.
Note that on $b\Omega$, $h=p$.

\begin{align*}
\rho_{z\bar z} =& r_{z\bar z}h +2\re[r_{\bar z}h_z]= r_{z\bar z}h +2\re[r_{\bar z}(K r_z + T_z)] \\
=& r_{z\bar z}h + 2K|r_z|^2 + 2\re[r_{\bar z}T_z]\\
\\
\rho_{w\bar w} =& r_{w\bar w}h+2\re[r_wh_{\bar w}]=r_{w\bar w}h +2\re[r_w(K r_{\bar w}+T_{\bar w})]\\
=& r_{w\bar w}h +2K|r_w|^2 +2\re[r_wT_{\bar w}]\\
\\
\rho_{z\bar w}=& r_{z\bar w}h +r_zh_{\bar w}+r_{\bar w}h_z = r_{z\bar w}h +r_z(K r_{\bar w} +T_{\bar w}) + r_{\bar w}(K r_z +T_z)\\
=& r_{z\bar w}h +2K r_z r_{\bar w} + r_zT_{\bar w} + r_{\bar w}T_z \\
\end{align*}

Note that $\rho_{w\bar w}$ is non-vanishing and positive for $K>0$ big enough.

The determinant of the complex Hessian is:
\begin{align*}
\sh_\rho=&\rho_{z\bar z}\rho_{w\bar w} - |\rho_{z\bar w}|^2\\
=&( r_{z\bar z}h +  2\re[r_{\bar z}T_z]+2K|r_z|^2 )( r_{w\bar w}h +2\re[r_wT_{\bar w}]+ 2K|r_w|^2)
-|r_{z\bar w}h + r_zT_{\bar w} + r_{\bar w}T_z+2K r_z r_{\bar w}|^2\\
=& (r_{z\bar z}h+  2\re[r_{\bar z}T_z])(2K|r_w|^2) + ( r_{w\bar w}h +2\re[r_wT_{\bar w}])(2K|r_z|^2 )\\
 &+( r_{z\bar z}h +  2\re[r_{\bar z}T_z])( r_{w\bar w}h +2\re[r_wT_{\bar w}]) +\cancel{ 4K^2|r_z|^2|r_w|^2}\\
 &-|r_{z\bar w}h + r_zT_{\bar w} + r_{\bar w}T_z|^2 \cancel{- 4K^2|r_z|^2|r_w|^2} -2\re[(2K r_{\bar z}r_w)(r_{z\bar w}h + r_zT_{\bar w} + r_{\bar w}T_z)]\\
=& 2Khr_{z\bar z}|r_w|^2+ \cancel{4K\re[r_{\bar z}T_z r_w r_{\bar w}]} + 2Khr_{w\bar w}|r_z|^2+\cancel{4K\re[r_w T_{\bar w}r_zr_{\bar z}]} + (pr)_{z\bar z}(pr)_{w\bar w}\\
& - |(pr)_{z\bar w}|^2- 4Kh\re[r_{\bar z}r_wr_{z\bar w}] \cancel{-4K\re[r_{\bar z}r_w r_z T_{\bar w}]} 
\cancel{- 4K\re[r_{\bar z}r_w r_{\bar w}T_z]}\\
=&2Kh \bigg[r_{z\bar z}|r_w|^2+r_{w\bar w}|r_z|^2 -2\re[r_{z\bar w}r_{\bar z}r_w]\bigg] +(pr)_{z\bar z}(pr)_{w\bar w}-|(pr)_{z\bar w}|^2\\
=& 2Kh\sL_r +\sh_{pr} \ .
\end{align*}
\end{proof}

\begin{remark}
If $A$ is the complex Hessian of $p\cdot r$ and $B$ is a complex Hessian of $r^2$, the proposition is equivalent to:
\\For $2\times2$ matrices $A$ and $B$,
\[det(A+B)= det A+det B +{det A}\cdot tr(A^{-1}B) \ ,\]
where  $det B=0$ and $det(A)tr(A^{-1}B)=h\sL_r$ when restricted to the boundary.
\end{remark}

Proposition \ref{Prop:Hessian} uncovers positivity of the Levi form that is ``hidden'' in the determinant of the complex Hessian. Thus the difficulty of producing a local plurisubharmonic defining function on the boundary lies in the construction of a real-valued function $T$ such that negativity of $\sh_{(1+T)r}$, the determinant of the Hessian of $(1+T)r$, can be controlled by the Levi form.
An immediate corollary to Proposition \ref{Prop:Hessian} is
\begin{theorem}\label{T:Hessiantresh}
$\Omega$ admits a local plurisubharmonic defining function on the boundary if and only if there exists a real-valued function $T$ and a constant $C>0$ such that $-\sh_{(1+T)r}\leq C\sL_r$.
\end{theorem}

\bigskip
The hardest term to control in the determinant of the complex Hessian $$\sh_r=r_{z\bar z}r_{w\bar w}-|r_{z\bar w}|^2$$ is the off-diagonal $-|r_{z\bar w}|^2$ term. Next, the necessary bounds, on the terms in the determinant of the complex Hessian, for the function to be plurisubharmonic are recorded.

\begin{lemma}\label{nece1}
Suppose $\rho$ is a local plurisubharmonic defining function. Then
\begin{enumerate}
\item $\frac{2h\sL_r}{|r_w|^2} +\frac{2\rho_{w\bar w}|r_z|^2}{|r_w|^2} \geq \rho_{z\bar z} $
\item $\rho_{z\bar z}\geq \frac{\rho_{w\bar w}|r_z|^2}{2|r_w|^2}-\frac{h\sL_r}{|r_w|^2}$ 
\item $\rho_{z\bar z}\geq \frac{h\sL_r}{2|r_w|^2} - \frac{\rho_{w\bar w}|r_z|^2}{|r_w|^2}$
\item $|\rho_{z\bar w}|^2\leq \frac{2\rho_{w\bar w}h\sL_r}{|r_w|^2} +\frac{2\rho_{w\bar w}^2|r_z|^2}{|r_w|^2}$ 
\end{enumerate}
As a consequence $\rho_{z\bar z},|\rho_{z\bar w}|^2 \leq C(\sL_r+|r_z|^2)$, or equivalently  $\rho_{z\bar z},|\rho_{z\bar w}|^2 =\mathcal O(\sL_r +|r_z|^2)$ in $U\cap b\Omega$ for some neighborhood $U$ of the origin.
\end{lemma}

\begin{proof}
Since $\rho$ is plurisubharmonic, the complex Hessian is positive  semidefinite for all vectors in $\C^2$. The complex Hessian for a vector
 $\langle 2^n r_w, -r_z\rangle$ is 
\begin{align*}
0\leq\sh (\langle 2^n r_w, -r_z\rangle) =&  \rho_{z\bar z}2^{2n}|r_w|^2 +\rho_{w\bar w}|r_z|^2 -2\cdot 2^n \re[\rho_{z\bar w} r_w r_{\bar z}]\\
=&  2^{2n}\rho_{z\bar z}|r_w|^2 +\rho_{w\bar w}|r_z|^2  + 2^n (\sL_\rho - \rho_{z\bar z}|r_w|^2 - \rho_{w\bar w}|r_z|^2)\\
=& 2^n(2^n-1) \rho_{z\bar z}|r_w|^2 + (1-2^n)\rho_{w\bar w}|r_z|^2 - 2^n \sL_\rho\\
=& 2^n \left [ (2^n-1) \rho_{z\bar z}|r_w|^2 +\frac{(1-2^n)}{2^n}\rho_{w\bar w}|r_z|^2 + \sL_\rho \right]
\end{align*}
Therefore
\[(2^n-1) \rho_{z\bar z}|r_w|^2 +\frac{(1-2^n)}{2^n} \rho_{w\bar w}|r_z|^2+ \sL_\rho \geq0\]
For $n=1,-1$, using the fact that $\sL_\rho=h\sL_r$ on $b\Omega$, inequalities (2) and (1) are obtained:
\[\rho_{z\bar z} |r_w|^2 - \frac12 \rho_{w\bar w}|r_z|^2 + h\sL_r\geq 0\]
\[ -\frac12\rho_{z\bar z}|r_w|^2 + \rho_{w\bar w}|r_z|^2 +h\sL_r \geq0\]
For (3), apply the complex hessian to the vector $\langle r_w, r_z\rangle$:
\begin{align*}
0\leq \sh(\langle r_w,r_z\rangle) =& \rho_{z\bar z}|r_w|^2+\rho_{w\bar w}|r_z|^2 +2\re[\rho_{z\bar w}r_{\bar z}r_w]\\
=& \rho_{z\bar z}|r_w|^2 + \rho_{w\bar w}|r_z|^2 -\sL_\rho +\rho_{z\bar z}|r_w|^2+\rho_{w\bar w}|r_z|^2\\
=&2\rho_{z\bar z}|r_w|^2 + 2\rho_{w\bar w}|r_z|^2 -\sL_\rho
\end{align*}
To obtain inequality (4), use (1) and $|\rho_{z\bar w}|^2 \leq \rho_{w\bar w}\rho_{z\bar z}$.
\end{proof}
	
\begin{remark}
The Example \ref{Ex:type4}  shows that we cannot improve the inequality (4) to $|\rho_{z\bar w}|^2\leq C \sL_r$.
\end{remark}

Using the bounds above, we prove the second key proposition.
\begin{proposition}
Suppose that $\rho=r\cdot h$ is a local plurisubharmonic defining function on the boundary. Then 
\begin{equation}
\frac{\partial}{\partial z} \log h = -\frac{\partial}{\partial z} \log r_{\bar w} +E \notag, 
\end{equation} where $|E|^2\leq C(\sL_r+|r_z|^2)$. 

If we write $h=1+Kr+T$, then 
\begin{equation} 
\frac{\partial}{\partial z} \log (1+T) = -\frac{\partial}{\partial z} \log r_{\bar w} +E \notag, 
\end{equation}
 where $|E|^2\leq C(\sL_r+|r_z|^2)$. 
\end{proposition}

\begin{proof}

We use the following fact in the subsequent calculations:
\[|A+B|^2=|A|^2+|B|^2+2\re[A\bar B] \geq \frac12|A|^2 - |B|^2 + \frac12|A|^2 +2|B|^2 - 2|A||B|\overset{\text{AM-GM}}{\geq}  \frac12|A|^2 - |B|^2\]

\[
C(\sL_r+|r_z|^2) \geq  |\rho_{z\bar w}|^2 = |r_{z\bar w}h + r_z h_{\bar w} + r_{\bar w}h_z|^2
\geq \frac12|r_{z\bar w}h+r_{\bar w}h_z|^2 - |r_z|^2|h_{\bar w}|^2
\]
\[
\tilde C(\sL_r+|r_z|^2) \geq|r_{z\bar w}h+r_{\bar w}h_z|^2 = |r_w||h|\left| \frac {h_z}{h}+ \frac{r_{z\bar w}}{r_{\bar w}}\right|^2
=|r_w||h|\left|\frac{\partial}{\partial z} \log h + \frac{\partial}{\partial z} \log r_{\bar w}\right|^2
\]
\[\hat C (\sL_r+|r_z|^2)\geq \left|\frac{\partial}{\partial z} \log h + \frac{\partial}{\partial z} \log r_{\bar w}\right|^2\]
Note that the logarithms are well defining in a small neighborhood of the origin, i.e. there is a branch cut consistent with $h$ and $r_{\bar w}$.

\bigskip
Write $h=1+Kr+T$, and notice that $h=1+T$ on $b\Omega$.
\[C(\sL_r +|r_z|^2)\geq  |r_{z\bar w}h +r_z h_{\bar w}+ r_{\bar w} (r_z+ T_z)|^2 \geq \frac12 |r_{z\bar w}(1+T)+r_{\bar w}T_z|^2 -|r_z|^2|r_{\bar w} +h_{\bar w}|^2\]
The conclusion follows from the calculation above.
\end{proof}

Therefore, we have obtained two necessary conditions on the multiplier $h=1+Kr+T$ to produce a local plurisubharmonic defining function on the boundary $\rho$. Namely:
\begin{theorem}\label{thm:nec}
Suppose that $\rho=r\cdot h$ is a local plurisubharmonic defining function on the boundary and $h=1+Kr+T$. Then
\begin{enumerate}
\item $\frac{\partial}{\partial z} \log (1+T) = -\frac{\partial}{\partial z} \log r_{\bar w} +E$ where $|E|^2\leq C(\sL_r+|r_z|^2)$ for some $C>0$.
\item $-\sh_{(1+T)r}\leq C\sL_r$ for some $C>0$.
\end{enumerate}
\end{theorem}

\section{Sufficient condition}
Suppose that there exists a real-valued function $T$ such that 
 \begin{equation}\label{E:logder4}
 \frac{\partial}{\partial z} \log (1+T) = -\frac{\partial}{\partial z} \log r_{\bar w} +E 
 \end{equation}
  where $|E|^2\leq C(\sL_r+|r_z|^2)$ for some $C>0$.

Let $p=1+T$.
The Equation \eqref{E:logder4} is equivalent to \[T_z= -\frac{p r_{z \bar w}}{r_{\bar w}}+pE\ . \]
Compute the second order derivatives:
\begin{align}\label{E:rhozz}
 (rp)_{z\bar z}&= pr_{z\bar z}+2\re[r_{\bar z}T_z]=pr_{z\bar z} +2\re[r_{\bar z}(-\frac{p r_{z \bar w}}{r_{\bar w}}+pE)] \notag\\ 
 &=p\left [r_{z\bar z}-2\re[r_{z\bar w}\frac{r_{\bar z}}{r_{\bar w}}]\right] +2p\re[r_{\bar z}E] \notag\\
 &= p\left[r_{z\bar z}-2\re[r_{z\bar w}\frac{r_{\bar z}}{r_{\bar w}}] +r_{w\bar w}\frac{|r_z|^2}{|r_w|^2}\right] - pr_{w\bar w}\frac{|r_z|^2}{|r_w|^2} +2p\re[r_{\bar z}E]\notag\\
 &= \frac{p}{|r_w|^2}\sL_r - pr_{w\bar w}\frac{|r_z|^2}{|r_w|^2} +2p\re[r_{\bar z}E]
\end{align}
and
\begin{align}\label{E:rhozw}
|(rp)_{z\bar w}|^2&= |pr_{z\bar w}+r_z T_{\bar w}+r_{\bar w}T_z|^2 =\left|pr_{z\bar w}+r_zT_{\bar w}+r_{\bar w}(-\frac{p r_{z \bar w}}{r_{\bar w}}+pE)\right|^2\notag\\
&=|r_zT_{\bar w} + pr_{\bar w}E|^2 = |r_z|^2|T_w|^2 + 2\re[pr_{\bar w}T_wr_{\bar z}E] +p^2|r_{\bar w}|^2|E|^2
\end{align} 
In particular, both $|(rp)_{z\bar z}|,|(rp)_{z\bar w}|^2\leq K(\sL_r+|r_z|^2)$ for some $K$.
Therefore $-\sh_{pr}=-\sh_{(1+T)r}\leq \tilde K(\sL_r+|r_z|^2)$.

\begin{theorem}\label{T:Levidom} 
Suppose that $|r_z|^2\leq C\sL_r$. Then there exists a local plurisubharmonic defining function on the boundary if and only if we can solve Equation \eqref{E:logder4}. Furthermore $\rho=(1+Kr+T)r$ is a local plurisubharmonic defining function on the boundary for some $K>0$.
\end{theorem}

\begin{proof}
We already showed that Equation \eqref{E:logder4} is a necessary condition. Under the additional assumption that there is a $T$ satisfying Equation \eqref{E:logder4}, $-\sh_{(1+T)r}\leq C(\sL_r+|r_z|^2) \leq \tilde C \sL_r$ for some $C,\tilde C>0$. Then by Theorem \ref{T:Hessiantresh} there exists $K>0$ such that $(1+Kr+T)r$ is a local plurisubharmonic defining function on the boundary.
\end{proof}

\subsection{$|r_z|^2\not\leq C\sL_r$ case}

If $|r_z|^2\not\leq C\sL_r$ for any $C>0$ in any neighborhood of the origin, then the terms involving $r_z$ and $E$ will play a significant role in determining if there is a local plurisubharmonic defining function on the boundary. 
The terms we do not have a desired control of are $-\frac{p r_{w\bar w}}{|r_w|^2}|r_z|^2$ and $2p \re[r_{\bar z}E]$ in the Equation \eqref{E:rhozz}. 
In the Equation \eqref{E:rhozw}, the problematic terms are $|r_z|^2|T_w|^2$, $2\re[pr_{\bar w}T_wr_{\bar z}E]$, and $p^2|r_{\bar w}|^2|E|^2$.

The terms involving $E$ can be improved by solving the Equation \eqref{E:logder4} to a higher order, for example, the $E$ terms satisfy $|E|^2\leq C\sL_r$ or even better. Secondly, if $\rho_{w\bar w}$ vanishes at the origin, the negativity of the determinant of the complex Hessian is easier to control.

On the other hand the terms involving just $|r_z|^2$ are harder to control. The lowest order terms of $T$, and therefore $T_w$, are forced by the Equation \eqref{E:logder4}. Thus there is no way of improving the bounds while still satisfying a necessary Equation \eqref{E:logder4}. However, $T_w$ and $r_{w\bar w}$ can be helpful. If $T_w$ and $r_{w\bar w}$ vanish to high order the determinant of the complex Hessian of $(1+T)r$ can still satisfy $-\sh_{(1+T)r}\leq C\sL_r$.

The Example \ref{Ex:type4} shows that $|r_z|^2\leq C\sL_r$ is not a necessary condition to having a local plurisubharmonic defining function on the boundary. Thus the following question remains open:
\begin{question}
Find a domain where $|r_z|^2$ is the obstruction to having a local plurisubharmonic defining function on the boundary.
\end{question}

\section{The Method}

\subsection{Basic strategy}
In this section we give an illustration of how to produce a local plurisubharmonic function on the boundary. 
\\Given a defining function $r$, our goal is to construct a real-valued $T$ that satisfies 
 \begin{equation}\label{E:logder5}
  \frac{\partial}{\partial z} \log (1+T) = -\frac{\partial}{\partial z} \log r_{\bar w} +E\ , 
  \end{equation} where  $|E|^2\leq C(\sL_r+|r_z|^2)$ for some $C>0$.
 
Write $r_w = \frac1{2i}+\tilde r_w$, where $\tilde r_w$ vanishes at the origin.
The Equation \eqref{E:logder5} is equivalent to 
\begin{align}\label{E:de1}
T_z&= -(1+T)\frac{r_{z\bar w}}{r_{\bar w}} + (1+T)E \notag\\
&=-(1+T)\frac{r_{z\bar w}}{-\frac1{2i}+\tilde r_{\bar w}}+(1+T)E \notag\\
&=  -(1+T)\frac{r_{z\bar w}}{-\frac1{2i}(1 -2i\tilde r_{\bar w})} + (1+T)E \notag\\
&= 2i(1+T)r_{z\bar w}(1+2i \tilde r_{\bar w}+(2i\tilde r_{\bar w})^2+...) +(1+T)E \notag \\
&= 2i r_{z\bar w}(1+ \mathcal O(T, \tilde r_{\bar w})) + (1+T)E
\end{align}
We wish to solve the Equation \eqref{E:logder5} up to the error terms $E$, where $|E|^2\leq C(\sL_r+|r_z|^2)$.
Therefore, any terms in \eqref{E:de1} whose modulus squared is bounded above by $C(\sL_r+|r_z|^2)$ for some $C>0$ can be omitted when solving this differential equation. In order to take full advantage, split $r_{z\bar w}=S_r+E$, where $|E|^2\leq C(\sL_r+|r_z|)^2$ and $S$ are the remaining terms which do not satisfy this bound. Thus disregarding the error terms $E$, the differential equation \eqref{E:de1} reduces to:
\begin{equation}\label{E:de2}
T_z = 2i S_r  (1+\mathcal O(T,\tilde r_{\bar w}))\ .
\end{equation}
We may view the remaining terms as a telescoping series. Since $T$ and $\tilde r_{\bar w}$ vanish at the origin they will produce terms of higher order, which we may ignore at the first pass.
Therefore we may further reduce the differential equation \eqref{E:de2} to:
\begin{equation}\label{E:de3}
T_z=2i S_r \ .
\end{equation}
Suppose a real-valued function $T$ solving \eqref{E:de3} exists. Let $\rho_1=(1+T)r$. Then
\begin{align*}
(\rho_1)_{z\bar w}=& r_{z\bar w}+r_z h_{\bar w} +r_{\bar w}h_z\\
=& (S_r+E)(1+T)+ r_z T_w +r_{\bar w} (2i S_r)\\
=& S_r+ S_rT +E(1+T) +r_zT_w +(-\frac1{2i}+\tilde r_{\bar w})2i S_r\\
=& S_rT+ E(1+T) +r_zT_w +2i\tilde r_{\bar w}S_r
\end{align*}
Notice the cancellation of the $S_r$.
Therefore, \begin{align*}
|(\rho_1)_{z\bar w}|&\leq |S_r||T|+ |E|(1+|T|) +|r_z||T_w|+2|r_{\bar w}||S_r|\\ 
&= |S_r|\cdot\mathcal O(|T|,|\tilde r_{\bar w}|)+ |E|(1+|T|)+|r_z||T_w|\ . \end{align*}
Using Cauchy-Schwarz we obtain:
\begin{align*}
|(\rho_1)_{z\bar w}|^2&\leq 3|S_r|\cdot\mathcal O(|T|,|\tilde r_{\bar w}|)+ 3|E|^2(1+|T|)^2+3|T_w|^2|r_z|^2 \\
&=  3|S_r|\cdot\mathcal O(|T|,|\tilde r_{\bar w}|) +\mathcal O(\sL_r+|r_z|^2)
\end{align*}

The new defining function $\rho_1$ need not be plurisubharmonic, but a necessary condition $|\rho_{z\bar w}|^2\leq C(\sL_r+|r_z|^2)$ is ``closer'' to being satisfied. This is so because the relevant terms $S_{\rho_1}$ of $(\rho_1)_{z\bar w}$ satisfy
\begin{equation}
 |S_{\rho_1}|\leq 3|S_r|\cdot\mathcal O(|T|,|\tilde r_{\bar w}|)=|S_r|\cdot \littleO(1) < |S_r|\ . 
 \end{equation}

We can repeat the above process to construct $\rho_2$, $\rho_3$,... 
If at any point $\rho_n$ is plurisubharmonic, we achieved our goal. On the other hand, if one differential equation $T_z=2i S_{\rho_n}$ that comes up in the construction cannot be solved, the domain does not have a local plurisubharmonic defining function on the boundary. 
This process may continue indefinitely, but we do not consider the question of convergence or plurisubharmonicity of the limit function in this paper.

\subsection{Strongly pseudoconvex domains}

Let $\Omega\subset\C^2$ be a strongly pseudoconvex domain. By the definition of strong pseudoconvexity, the Levi form $\sL_r>0$ for every point on the boundary. In particular, the Levi form at the origin, $\sL_r(0,0)=L>0$, is positive. By the continuity of the Levi form there is a neighborhood of the origin $U$ on which $\sL_r>\frac L2$. To produce a local plurisubharmonic defining function we need to solve the Equation \eqref{E:logder5}. However, we notice that on $U$, $|r_{z\bar w}|\leq C=\mathcal O(1)=\mathcal O(\sL)$. Therefore the Equation \eqref{E:logder5} reduces to 
\[\frac{\partial}{\partial z} \log (1+T) = 0 +E \ ,\]
which is solved by $T=0$. 

Also, on $U$,  $|r_z|^2\leq \tilde C=\mathcal O(1)=\mathcal O(\sL_r)$. Therefore by Theorem $\ref{T:Levidom}$, $(1+Kr)r$ is a local plurisubharmonic defining function on the boundary for some $K>0$.

\subsection{Example of D'Angelo type $4$}

\begin{example}\label{Ex:type4}
Let $r_A(z,w)= \im w +|z|^4+100|z|^6 +4\re z \re w - A(\re w)^2$. 
\end{example}

\begin{lemma}
$\Omega_A=\{r_A<0\}$ is pseudoconvex in a (small) neighborhood of the origin if and only if $A\geq8$.
\end{lemma}

\begin{proof}
A direct computation shows that in a sufficiently small neighborhood of the origin
\begin{align*}
\sL_{r_A} \geq & \frac12(\im z)^2+ 100|z|^4 +(\re z)^2 +2A(\re w)^2 -8\re z\re w\\
=&  \frac12(\im z)^2+ 100|z|^4 + \epsilon(\re z)^2 + (2A-\frac{16}{1-\epsilon})(\re w)^2 + (\sqrt{1-\epsilon}\re z- \frac4{\sqrt{1-\epsilon}}\re w)^2\\
\geq& \frac12(\im z)^2 + \epsilon(\re z)^2 + (2A-\frac{16}{1-\epsilon})(\re w)^2 + 100|z|^4
\end{align*}
for any $1>\epsilon\geq0$. This shows that $\Omega_A$ is pseudoconvex if $A\geq 8$. (Actually, $\Omega_A$ is pseudoconvex if and only if $A\geq8$.)
\end{proof}

\begin{lemma}
Let $A\geq8$. The defining function $r_A$ is not plurisubharmonic in any neighborhood of the origin.
\end{lemma}

\begin{proof}
A necessary condition for $r_A$ being plurisubharmonic is $(r_A)_{w\bar w}\geq0$ in some neighborhood of the origin. However, $(r_A)_{w\bar w}=-\frac{A}{2}<0$. 
\end{proof}

\begin{lemma}
Let $A>8$. Then $\rho_1=r(1-4\im z+Kr)$ is a local plurisubharmonic defining function on the boundary for some $K>0$.
\end{lemma}

\begin{proof}
Since $A>8$, we can pick $\epsilon>0$ small enough such that 
\begin{equation}\label{E:Levi-rA}
\sL_{r_A}\geq \frac12(\im z)^2+ \epsilon (\re z)^2+ \epsilon (\re w)^2\geq \epsilon |z|^2 +\epsilon (\re w)^2 \ .
\end{equation}
Notice that
\begin{align*}
|(r_A)_z|^2=& \left| 2z|z|^2+300z|z|^4+2\re w\right|^2 \leq 2\left| 2z|z|^2+300z|z|^4\right|^2 +8(\re w)^2 \\
=&8(\re w)^2+\mathcal O(|z|^6)\leq \frac8\epsilon \sL_{r_A} \ .
\end{align*}
Thus by Theorem \ref{T:Levidom} $\rho=r(1+ T +Kr)$ is a local plurisubharmonic defining function for some $K>0$ if and only if $T$ solves the differential equation \eqref{E:logder5}. Furthermore, \eqref{E:Levi-rA} shows that any non-constant term in the equation \eqref{E:logder5} is an error term $E$.

The first differential equation we need to solve is \[T_z= 2i S_{r_A}=2i(1) \ . \]
It is easy to check that $T=-4\im z$ is a solution to $T_z=2i S_{r_A}$ and the differential equation \eqref{E:logder5} is satisfied up to error terms $E$.
Therefore $\rho=\rho_1=r(1-4\im z+Kr)$ is a local plurisubharmonic defining function on the boundary for some $K>0$.
\end{proof}

\begin{lemma}
Let $A=8$. Then $\rho_2=r(1-4\im z - 8(\re z)^2 + 8(\im z)^2 + Kr)$ is a local plurisubharmonic defining function on the boundary for some $K>0$.
\end{lemma}

\begin{proof}
The Levi form is given by 
\[\sL_{r_8}\geq \frac12(\im z)^2 + (\re z-4\re w)^2 +|z|^4\ ,\]
and 
\[|(r_8)_z|^2=4(\re w)^2+ 4\re w\re\left[2z|z|^2+300z|z|^4\right]+ \left|2z|z|^2+300z|z|^4\right|^2\ .\]
Notice that on $\{\im z=0, \re z= 4\re w\}$ the Levi form vanishes to order $4$, while $|(r_8)_z|^2$ vanishes to order $2$. Therefore $|(r_8)_z|^2\not\leq C\sL_{r_8}$ for any $C>0$ in any neighborhood of the origin. Thus Theorem \ref{T:Levidom} cannot be applied.

Moreover, direct computation shows that $\rho_1=r(1-4\im z+Kr)$ is not plurisubharmonic for any $K>0$. In this case we need to solve the differential equation \eqref{E:logder5} to higher order. 
The second differential equation then becomes
\[(T_2)_z = 2i\left((\rho_1)_{z\bar w}\right) = 2i(-4\im z+4i\re z- 16i\re w) = -8i \im z - 8\re w +32\re w\ . \]
The solution is $T_2 = 8(\im z)^2 - 8(\re z)^2 +64\re z\re w$. However, notice that $64\re z\re w$ is a multiple of a term in the defining function $r$. 
Thus, the contributions of this term already come from the $Kr$ term in the multipler function $h$. 
Therefore we can omit $64\re z\re w$ from $T$ to produce a simpler multiplier function $h$. A direct computation shows that $\rho=\rho_2= r(1-4\im z +8(\im z)^2- 8(\re z)^2 +Kr)$ is a local plurisubharmonic defining function on the boundary for some $K>0$.

\end{proof}

\section{Domains in $\C^n$}

Let $\Omega={r(z_1,....,z_n)<0}\subset \C^n$.
The complex tangent space $T^{1,0}_p b\Omega$  to $b\Omega$ at $p$ is defined to be
all $\xi=(\xi_1,...,\xi_n)\in\C^n$ satisfying $$\sum_{j=1}^n\frac{\partial r}{\partial z_j}(q)\xi_j=0.$$
$\Omega$ is pseudoconvex if

$$\sL_r(\xi)=\sum_{j,k=1}^n\frac{\partial^2 r}{\partial z_j\partial\bar z_k}(q)\xi_j\bar\xi_k\geq 0\qquad\text{ for all } \xi\in T^{1,0}_pb\Omega,$$
Similarly, $r$ is plurisubharmonic if
$$\sL_r(\xi)=\sum_{j,k=1}^n\frac{\partial^2 r}{\partial z_j\partial\bar z_k}(q)\xi_j\bar\xi_k\geq 0\qquad\text{ for all } \xi\in \C^n$$

Choose a coordinate system such that the defining function $r$ is of the form \[r(z_1,...,z_{n-1}, w)=\im w +F(z_1,\bar z_1,...,z_{n-1},\bar z_{n-1}, \re w,\im w)\ , \] where terms of $F$ are of degree at least $2$.

Let $v_j=\langle 0,...,0, r_w, 0,...0, -r_{z_j}\rangle$, where $r_w$ is in the $j$-th coordinate for $j=1,...,n-1$. Notice that $v_1,...,v_{n-1}$ form a basis  for the tangent space $T^{1,0}_p b\Omega$ to $b\Omega$ at $p$. Let $$\sL_r(v_j)=r_{{z_j}\bar{z_j}}|r_w|^2+r_{w\bar w}|r_{z_j}|^2 -2\re[r_{{z_j}\bar w}r_wr_{\bar z_j}]$$ be the Levi form acting on $v_j$.
Let $\sh^j_f=det\begin{pmatrix}
f_{z_j\bar{z_j}} & f_{z_j \bar w}\\
f_{\bar{z_j} w} & f_{w\bar w}
\end{pmatrix}$ be the determinant of the corresponding $2\times2$ minor of the complex Hessian of $f$.

Similar proofs as in $\C^2$ case, by considering the appropriate $2\times 2$ minor of the complex Hessian matrix, give the following analogs:

\begin{proposition}\label{Prop:nHessian}
For $j=1,...,n-1$
\[\sh^j_{(1+Kr+T)r}=2Kh\sL_r(v_j)+\sh^j_{(1+T)r}\]
\end{proposition}

\begin{lemma}\label{nece1}
Suppose $\rho=r\cdot h$ is a local plurisubharmonic defining function on the boundary.
\begin{enumerate}
\item $\frac{2h\sL_r(v_j)}{|r_w|^2} +\frac{2\rho_{w\bar w}|r_{z_j}|^2}{|r_w|^2} \geq \rho_{z_j\bar z_j} $
\item $\rho_{z_j\bar z_j}\geq \frac{\rho_{w\bar w}|r_{z_j}|^2}{2|r_w|^2}-\frac{h\sL_r(v_j)}{|r_w|^2}$ 
\item $\rho_{z_j\bar z_j}\geq \frac{h\sL_r(v_j)}{2|r_w|^2} - \frac{\rho_{w\bar w}|r_{z_j}|^2}{|r_w|^2}$
\item $|\rho_{z_j\bar w}|^2\leq \frac{2\rho_{w\bar w}h\sL_r(v_j)}{|r_w|^2} +\frac{2\rho_{w\bar w}^2|r_{z_j}|^2}{|r_w|^2}$ 
\end{enumerate}

As a consequence $\rho_{z_j\bar {z_j}},|\rho_{z_j\bar w}|^2 \leq C(\sL_r(v_j)+|r_{z_j}|^2)$, or equivalently  $\rho_{z_j\bar {z_j}},|\rho_{z_j\bar w}|^2 = \mathcal O(\sL_r(v_j)+|r_{z_j}|^2)$ in $U\cap b\Omega$ for some neighborhood $U$ of the origin .
\end{lemma}

\begin{proposition}\label{Prop:nlogder}
Suppose that $\rho=r\cdot h$ is a local plurisubharmonic defining function on the boundary. Then for $j=1,...,n-1$
\begin{equation}\label{E:nlogder}
\frac{\partial}{\partial z_j} \log h = -\frac{\partial}{\partial z_j} \log r_{\bar w} +E_j , \end{equation}
where $|E_j|^2\leq C(\sL_r(v_j)+|r_{z_j}|^2)$. 
\end{proposition}
The same mutiplier function $h$ needs to satisfy the differential equation \eqref{E:nlogder} for all $j=1,...,n-1$.

\section{Concluding remarks}

A similar problem is of interest in the real setting as well: can we find a convex defining function (real Hessian is positive semi-definite for all real vectors) for convex domains (real Hessian of a defining function is positive semi-definite for tangent vectors) in $\R^n$. The statements and proofs are completely analogous. 
\begin{proposition}
Suppose that $\rho=r\cdot h$ is a convex defining function for $$\Omega=\{r(x_1,...,x_{n-1},y)<0\}\subset \R^n$$ with $r_y\neq0$. Then for $j=1,...,n-1$
\[\frac{\partial}{\partial x_j} \log h = -\frac{\partial}{\partial x_j} \log r_{y} +E_j , \] where $|E_j|^2\leq C(\tilde\sL_r(v_j)+|r_{x_j}|^2)$ and $\tilde \sL_r(v_j)$ is the ``real analog of the Levi form''.
\end{proposition}
However the real setting is ``easier'' in the following sense: since $r_y$ is real, we may take $h=1+Kr+r_y$ to be the multiplier function. 

This recovers, from a different view-point, a result from \cite{HerMcN09} that every convex domain in $\R^n$ has a defining function whose Hessian is positive semi-definite in a neighborhood of the boundary of the domain.

A further question is considering extending the positivity of the Hessian past the boundary and inside the domain. 
In a similar fashion to Proposition \ref{Prop:Hessian}, we obtain 
\begin{align*}\sh_{rh}=& \sh_{rp} +4K^2r^2 \sh_r +4K^2 r\sL_r + 2K \sL_{rp} +2Kr \left[(rp)_{z\bar z}r_{w\bar w} +(rp)_{w\bar w}r_{z\bar z} -2\re[(rp)_{z\bar w}r_{\bar z w}]\right] \\
=&\sh_{rp} +4K^2r^2 \sh_r +4K^2 r\sL_r + 2K p \sL_r+2K r \sL_p \\
&+2Kr [(rp)_{z\bar z}r_{w\bar w} +(rp)_{w\bar w}r_{z\bar z} -2\re[(rp)_{z\bar w}r_{\bar z w}]]
\end{align*}
We plan to consider this in a future paper.

\bibliographystyle{acm}
\bibliography{BibMaster}

\begin{thebibliography}{10}

\bibitem{Behrens85}
{\sc Behrens, M.}
\newblock Plurisubharmonic defining functions of weakly pseudoconvex domains in
  {${\bf C}\sp 2$}.
\newblock {\em Math. Ann. 270}, 2 (1985), 285--296.

\bibitem{BoasStraube90}
{\sc Boas, H.~P., and Straube, E.~J.}
\newblock Equivalence of regularity for the {B}ergman projection and the
  $\overline \partial$-{N}eumann operator.
\newblock {\em Manuscripta Math. 67}, 1 (1990), 25--33.

\bibitem{BoaStr91}
{\sc Boas, H.~P., and Straube, E.~J.}
\newblock Sobolev estimates for the {$\overline\partial$}-{N}eumann operator on
  domains in {${\bf C}^n$} admitting a defining function that is
  plurisubharmonic on the boundary.
\newblock {\em Math. Z. 206}, 1 (1991), 81--88.

\bibitem{BoasStraube1993}
{\sc Boas, H.~P., and Straube, E.~J.}
\newblock De rham cohomology of manifolds containing the points of infinite
  type, and sobolev estimates for the {$\overline\partial$}-{N}eumann problem.
\newblock {\em The Journal of Geometric Analysis 3}, 3 (May 1993), 225--235.

\bibitem{BoasStraube99}
{\sc Boas, H.~P., and Straube, E.~J.}
\newblock Global regularity of the {$\overline\partial$}-{N}eumann problem: a
  survey of the {$L^2$}-{S}obolev theory.
\newblock In {\em Several complex variables ({B}erkeley, {CA}, 1995--1996)},
  vol.~37 of {\em Math. Sci. Res. Inst. Publ.} Cambridge Univ. Press,
  Cambridge, 1999, pp.~79--111.

\bibitem{Demailly1987}
{\sc Demailly, J.-P.}
\newblock Mesures de monge-amp{\`e}re et mesures pluriharmoniques.
\newblock {\em Mathematische Zeitschrift 194}, 4 (Dec 1987), 519--564.

\bibitem{DieFor77-2}
{\sc Diederich, K., and Forn\ae\hspace{-.00cm}ss, J.~E.}
\newblock Pseudoconvex domains: bounded strictly plurisubharmonic exhaustion
  functions.
\newblock {\em Invent. Math. 39}, 2 (1977), 129--141.

\bibitem{DieFor77-1}
{\sc Diederich, K., and Forn{\ae}ss, J.~E.}
\newblock Pseudoconvex domains: an example with nontrivial {N}ebenh\"ulle.
\newblock {\em Math. Ann. 225}, 3 (1977), 275--292.

\bibitem{Fornaess79}
{\sc Forn{\ae}ss, J.~E.}
\newblock Plurisubharmonic defining functions.
\newblock {\em Pacific J. Math. 80}, 2 (1979), 381--388.

\bibitem{herbigfornaess07}
{\sc Forn{\ae}ss, J.~E., and Herbig, A.-K.}
\newblock A note on plurisubharmonic defining functions in $\mathbb{C}^2$.
\newblock {\em Math. Z. 257\/} (2007), 769--781.

\bibitem{herbigfornaess08}
{\sc Forn{\ae}ss, J.~E., and Herbig, A.-K.}
\newblock A note on plurisubharmonic defining functions in $\mathbb{C}^n$.
\newblock {\em Math. Ann. 342\/} (2008), 749--772.

\bibitem{Harrington07}
{\sc Harrington, P.}
\newblock The order of plurisubharmonicity on pseudoconvex domains with
  lipschitz boundaries.
\newblock {\em Mathematical Research Letters 15\/} (01 2007).

\bibitem{HerMcN09}
{\sc Herbig, A.-K., and McNeal, J.~D.}
\newblock Convex defining functions for convex domains.
\newblock {\em J. Geom. Anal. 22}, 2 (2012), 433--454.

\bibitem{HerMcN2012}
{\sc Herbig, A.-K., and McNeal, J.~D.}
\newblock Oka’s lemma, convexity, and intermediate positivity conditions.
\newblock {\em Illinois J. Math. 56}, 1 (2012), 195--211.

\bibitem{KerzmanRosay1981}
{\sc Kerzman, N., and Rosay, J.-P.}
\newblock Fonctions plurisouscharmoniques d'exhaustion born{\'e}es et domaines
  taut.
\newblock {\em Mathematische Annalen 257}, 2 (Oct 1981), 171--184.

\bibitem{Kohn99}
{\sc Kohn, J.~J.}
\newblock Quantitative estimates for global regularity.
\newblock In {\em Analysis and geometry in several complex variables ({K}atata,
  1997)}, Trends Math. Birkh\"auser Boston, Boston, MA, 1999, pp.~97--128.

\bibitem{Liu19-1}
{\sc Liu, B.}
\newblock The {D}iederich-{F}ornaess index i: For domains of non-trivial index.
\newblock {\em Advances in Mathematics 353\/} (2019), 776 -- 801.

\bibitem{Liu19-2}
{\sc Liu, B.}
\newblock The {D}iederich-{F}ornaess index ii: For domains of trivial index.
\newblock {\em Advances in Mathematics 344\/} (2019), 289 -- 310.

\bibitem{morrowrossi1975}
{\sc Morrow, J., and Rossi, H.}
\newblock Some theorems of algebraicity for complex spaces.
\newblock {\em J. Math. Soc. Japan 27}, 2 (04 1975), 167--183.

\end{thebibliography}

\end{document}